\documentclass[a4paper,11pt]{article}

\usepackage[margin=3cm]{geometry}
\usepackage{setspace}
\usepackage{amsmath}
\usepackage{amssymb}
\usepackage{amsthm}

\usepackage{psfrag}
\usepackage{graphicx}
\usepackage{color}

\usepackage{alltt}

\usepackage{enumerate}
\usepackage[caption=false]{subfig}

\def\bG{\bar{\Gamma}}
\def\bS{\bar{\Sigma}}
\def\bV{\bar{V}}
\def\bU{\bar{U}}

\def\tg{\tilde{\gamma}}
\def\ts{\tilde{\sigma}}
\def\tv{\tilde{v}}
\def\tu{\tilde{u}}

\newcounter{dummy}
\newtheorem{theorem}{Theorem}
\newtheorem{lemma}{Lemma}[section]
\newtheorem{proposition}{Proposition}[section]
\newtheorem{definition}{Definition}[section]
\newtheorem{remark}{Remark}[section]

\newtheorem{step}{Step}[dummy]
\numberwithin{step}{dummy}

\begin{document}
\title{Existence of localizing solutions in plasticity \\ via the geometric singular perturbation theory}
\author{Min-Gi Lee\footnotemark[1]\  \footnotemark[2]
\and Athanasios Tzavaras\footnotemark[1]}
\date{}

\maketitle
\renewcommand{\thefootnote}{\fnsymbol{footnote}}
\footnotetext[1]{Computer, Electrical and Mathematical Sciences \& Engineering Division, King Abdullah University of Science and Technology (KAUST), Thuwal, Saudi Arabia}
\footnotetext[2]{Corresponding author : \texttt{mingi.lee@kaust.edu.sa}}
\renewcommand{\thefootnote}{\arabic{footnote}}

\begin{abstract}
Shear bands are narrow zones of intense shear observed during  plastic deformations of metals at high strain rates. Because they often precede rupture, 
their study attracted attention as a mechanism of material failure.  
Here, we aim to reveal the onset of localization into shear bands using a simple model developed from viscoplasticity.
We  exploit the properties of scale invariance of the model to construct a family of self-similar focusing solutions that capture 
the nonlinear mechanism of  shear band formation. The key step is to de-singularize a reduced  system of singular ordinary differential equations 
and reduce the problem into the construction of a heteroclinic orbit for an autonomous system  of three first-order equations. 
The associated  dynamical system has  fast and slow time scales, forming a singularly perturbed problem.  
Geometric singular perturbation theory is applied to this problem to achieve an invariant surface. The flow on
the invariant surface  is analyzed via the Poincar\'{e}-Bendixson theorem to construct a heteroclinic orbit. 
\end{abstract}

%

\section{Introduction} \label{sec:intro}
We consider  the system of partial differential equations
\begin{align}
 \begin{aligned}
 \gamma_t&= v_x, \\
 v_t &= \big( \gamma^{-m} v_x^n\big)_x,
 \end{aligned}\label{intro-system}
\end{align}
where $(x,t)\in \mathbb{R} \times \mathbb{R}^+$,
which describes shear motions of a viscoplastic material and in terms of classification 
belongs to the class of hyperbolic-parabolic systems.
Here, $\gamma$ is the plastic strain, $v$ is the velocity in the shearing direction,
and $m,n>0$ are material parameters. 
The system \eqref{intro-system} is a model from viscoplasticity that serves as a simplified model to understand the problem of 
shear band formation in metals deformed at high strain rates (see Section  \ref{sec2}).  
The yield relation $\sigma = \gamma^{-m} \gamma_t^n$ characterizes the viscoplastic nature of materials: 
$\gamma^{-m}$ accounts for plastic (net) strain softening and $\gamma_t^n$ for strain-rate hardening. The latter term models dissipation 
by momentum diffusion manifested by mathematical viscosity in the form present in non-Newtonian fluids. 

For $n=0$, the system \eqref{intro-system}  is elliptic in the $t$-direction and exhibits {\it Hadamard instability} - the catastrophic growth of oscillations for the linearized  
initial value problem - induced by the (net) strain-softening response.  But when $n>0$, the viscosity competes against this ill-posedness.  The combination
of the destabilizing effect of strain softening and the stabilizing effect of strain-rate hardening is conjectured to lead to localization of the strain in narrow zones  
called shear bands \cite{zener_effect_1944, clifton_rev_1990}. Their formation is helpful for explaining mechanisms of material failure;  we refer to
\cite{zener_effect_1944, clifton_critical_1984, shawki_shear_1989, clifton_rev_1990, wright_survey_2002, bertsch_effect_1991, tzavaras_nonlinear_1992}
and to Section  \ref{sec2} for further details on this problem.

To set the localization problem in the language of mathematical analysis, observe that
\eqref{intro-system} admits a class of solutions,  that are valid for any values of the parameters $m$ and $n$ and describe uniform shearing 
\begin{align}
\label{intro-uss}
 v_s(x) = x, \quad \gamma_s(t) = t + \gamma_0, \quad \sigma_s (t)  = (t + \gamma_0)^{-m}.
\end{align}
The issue then becomes to examine whether small perturbations of the uniform shearing solutions develop nonuniformities that go astray or whether nonuniformities
get suppressed resulting into  stable response. In the regime $n > m$, both  linearized and nonlinear analyses
\cite{tzavaras_plastic_1986, fressengeas_instability_1987, shawki_shear_1989, tzavaras_strain_1991, tzavaras_nonlinear_1992} indicate that the
uniform shearing solutions are stable.
On the complementary region $m > n$, an analysis of the linearized system of relative perturbations \cite{fressengeas_instability_1987, molinari_analytical_1987, tzavaras_nonlinear_1992} indicates instability of the uniform shearing solutions.

In this work, we aim to reveal the subtle mechanism of shear band formation in the nonlinear regime and to construct a class of self-similar solutions
that exhibit localization in the regime $m > n$. We exploit the invariance properties of the system \eqref{intro-system} and seek self-similar
solutions of the form
\begin{equation} \label{intro-ss}
\begin{aligned}
 \bar{ \gamma }(t,x) &=  (t + 1)^a\bG((t+1)^ \lambda x),\\
 \bar{v}(t,x)&= (t + 1)^b\bV((t+1)^ \lambda x) \, , 
\end{aligned}
\end{equation}
where  $\xi=(t+1)^ \lambda x$ is the similarity variable and $\lambda > 0$ is a parameter.  The reader should note that the usual form of self-similar solutions
for parabolic problems are generated for values of the parameter $\lambda < 0$ and capture the spreading effect associated with parabolic behavior. 
By contrast, we insist here on $\lambda > 0$ and study the existence of solutions focusing around the line $x =0$ as time proceeds. 
This idea for constructing localizing solutions is proposed  in \cite{katsaounis_emergence_2014} for a non-Newtonian fluid with temperature-dependent viscosity 
and in \cite{KLT_2016} for \eqref{intro-system} with $m=1$.

The parameters $a$ and $b$ are selected by
\begin{equation}
\label{defab}
 a^{\lambda,m,n}= \frac{2-n}{1+m-n} + \frac{2 \lambda}{1+m-n}, \quad b^{\lambda,m,n}= \frac{1-m}{1+m-n} + \frac{1-m+n}{1+m-n} \lambda
\end{equation}
and the profiles $(\bG, \bV)$ are constructed by solving an initial value problem for a singular system of ordinary differential equations
\begin{align} 
\label{intro-sseqns}
\begin{aligned}
 a^{\lambda,m,n}\bG + \lambda \xi \bG_\xi &= \bV_\xi,\\
 b^{\lambda,m,n}\bV + \lambda \xi \bV_\xi &= \big(\bG^{-m}\bV_\xi^n\big)_\xi,
\end{aligned}
\\
\label{intro-ssdata}
 \bG(0) = \bG_0  > 0 \, , \quad \bV_{\xi} (0)  = \bU(0) = \bU_0  > 0 \, , 
 \end{align}
where  $\bU(\xi) = \bV_\xi(\xi)$ and $\bG_0$ and $\bU_0$ are given positive parameters. As seen from $\eqref{intro-sseqns}_1$ at $\xi=0$, 
$$
 a^{\lambda,m,n} \Gamma_0 = U_0
 $$
and thus  two out of the parameters $\lambda$, $\bG_0$ and $\bU_0$ fix the third.

There is no sufficiently general theory that guarantees the existence of solutions for such singular initial value problems and the construction is
usually based on a case-by-case analysis.
Remarkably, the invariance properties of the system \eqref{intro-sseqns} allows the  de-singularization of the system \eqref{intro-sseqns}
(see \eqref{eq:tvars} and \eqref{eq:auto0}). Furthermore, a nonlinear change of variables (see \eqref{eq:ratios}) leads to reformulating the problem
into an autonomous system of three first-order equations
{\small
\begin{align} 
 \dot{p} &=p\Big( ~~~~~~\frac{1}{ \lambda }\big(r - \frac{2-n}{1+m-n}\big) - \frac{1-m+n}{1+m-n} &+&1-q- \lambda p r\Big), \nonumber \\
 \dot{q} &=q\Big(                                                                          &+&1-q- \lambda p r\Big) + b^{\lambda,m,n}pr,\label{intro-pqrsystem}\\
 n\dot{r}&=r\Big( \frac{m-n}{ \lambda }\big(r - \frac{2-n}{1+m-n}\big) + \frac{1-m+n}{1+m-n} &-&1+q+ \lambda p r\Big). \nonumber
\end{align}}
Moreover, the question of existence of a solution $\big(\bV,\bG)$ to \eqref{intro-sseqns}, \eqref{intro-ssdata} is reformulated 
to that of the construction of a suitable heteroclinic orbit to \eqref{intro-pqrsystem}. The difficulty with the construction of
such heteroclinics originates from the dimensionality of the system \eqref{intro-pqrsystem}.
In \cite{KLT_2016},  we considered a system related to the case $m=1$ and numerically constructed  the heteroclinic orbit.

The main result of this work is that by exploiting geometric singular perturbation theory, one may  construct the heteroclinic orbit.
As $n$ is a small parameter,  the system \eqref{intro-pqrsystem} admits both fast and slow time scales. Problems with multiple time scales are habitually found 
in multiple contexts, and one  gets a clear picture of the problem by analyzing the geometric picture in the phase space via geometric 
singular perturbation theory \cite{fenichel_persistence_1972, fenichel_geometric_1979, jones_geometric_1995, wiggins_normally_1994, KUEHN_2015}.
 Among many successful applications of the theory, there are several examples 
 \cite{gasser_geometric_1993, freistuhler_spectral_2002, xiao_stability_2003, SS_2004,ghazaryan_traveling_2007} of application to the resolution
 of viscous wave fans in hyperbolic conservation laws. We present a novel application of the method to
 analyze the nonlinear competition of Hadamard instability with viscosity effected by strain-rate hardening in dynamic plasticity.

The paper is organized as follows: In Section 2, we briefly explain the background and describe the mechanical problem studied in this paper. 
In Section 3,  a class of focusing self-similar solutions are introduced  and the associated system of singular ordinary differential equations is derived. 
The problem  is then reduced into the construction of a heteroclinic orbit of an associated autonomous system. 
In Section 4,  the phase space analysis is carried out for this autonomous system. 
We identify its equilibria,  study their dynamical nature and use mechanical considerations to  select the targeted heteroclinic orbit. 
In Section 5,  we  construct a normally hyparbolic invariant manifold using geometric singular perturbation theory (reviewed in the appendix) and study the
dynamical system restricted on  that manifold to establich the existence of the heteroclinic orbit.  
 The emerging two-parameter family of localizing solutions
to the system \eqref{intro-system} is outlined in section 6 (see \eqref{eq:sssol}), where various properties, such as the  range of parameters 
and growth behavior of the solution are scrutinized.

\section{Background and a Description of the Problem}
\label{sec2}
We investigate the formation of  shear bands during the high-strain-rate shear deformation of metals. 
At high strain rate, shear can accumulate  in narrow zones, often leading to rupture. 
Several works have focused on this behavior to explain material failure 
\cite{zener_effect_1944, clifton_critical_1984,  shawki_shear_1989, clifton_rev_1990, wright_survey_2002}.
In experimental investigations of deformations of steels at high strain-rates,  observations of shear bands are  typically associated 
with strain softening response -- past a critical strain -- of the measured stress-strain curve \cite{clifton_critical_1984}. 
It was proposed by Zener and Hollomon \cite{zener_effect_1944}, and further developed by Clifton et al  \cite{clifton_critical_1984, clifton_rev_1990},
that the effect of the deformation speed is twofold:
An increase in the deformation speed changes the deformation
conditions from isothermal to nearly adiabatic, and the combined effect of thermal softening and strain hardening of metals may produce 
a net softening response. On the other hand, strain-rate hardening has an effect {\it per se}, inducing momentum diffusion and playing a stabilizing role.

Modeling this mechanism requires to consider the effect of the energy equation. Nevertheless, a simper model has been proposed in order to
assess the effect of (net) strain softening
response in shear motions of a  viscoplastic model  to provide quantitative analysis in the problem of
localization \cite{HN77,tzavaras_plastic_1986,tzavaras_nonlinear_1992}. The system 
\begin{equation} \label{eq:system0}
 \begin{aligned}
   v_t &= \sigma(\gamma,\gamma_t)_x, \\ 
   \gamma_t&=v_x.
 \end{aligned}
\end{equation}
models shear motions of a viscoplastic material exhibiting strain softening  and strain-rate hardening
$$
 \frac{\partial\sigma}{\partial\gamma}<0 \, , \quad   \frac{\partial\sigma}{\partial \gamma_t } > 0 \, , 
$$
respectively. 
\eqref{eq:system0} consists of momentum conservation and kinematic compatibility, where $v$ is the velocity in the shearing direction, 
$\gamma$ is the plastic shear strain,  and $\sigma$ is the shear stress. The model describes a specimen situated on the  $xy$-plane that shears in the $y$-direction. 
The simplifying assumption here is that the (plastic) strain $\gamma$ and the strain rate $\gamma_t$ solely characterize the stress yield relation
$ \sigma = \sigma(\gamma,\gamma_t)$.  The purpose of our study is to analyze a type of instability emerging out of the competition between 
strain softening and strain-rate hardening. Our study is not limited to capturing the shear bands but also seeks to explain this viscoplastic instability mechanism
and the emergence of organized structures out of this competition.
Early mathematical treatments of the initial value problem can be found in \cite{tzavaras_plastic_1986,tzavaras_strain_1991}.

A simple choice for $\sigma(\gamma,\gamma_t)$  is given by 
\begin{equation} \label{eq:constitution1}
 \sigma = \sigma(\gamma,\gamma_t)=\varphi(\gamma)\gamma_t^n
\end{equation}
(see \cite{molinari_analytical_1987, tzavaras_strain_1991}) 
where $\varphi'(\gamma)<0$ and $n>0$ is the rate sensitivity parameter which is typically very small \cite{shawki_shear_1989}. 
When  $n=0$, then $\sigma=\varphi(\gamma)$ and the condition of strain-softening $\varphi'(\gamma)<0$ implies that
the system \eqref{eq:system0} is elliptic in the $t$-direction. Then the initial value problem exhibits
 Hadamard instability, that is the linearized problem exhibits catastrophic growth in the high frequency oscillatory modes.  
 But when  $n > 0$, the  effect of viscosity  competes against this instability.

The model \eqref{eq:system0}-\eqref{eq:constitution1} admits the uniform shearing solutions, 
\begin{align} \label{eq:uss}
 \begin{aligned}
  v_s(x) &= x, \\
  u_s(t) &= \partial_x v_s(x,t) = 1,\\
  \gamma_s(t) &= t + \gamma_0, \quad \text{$\gamma_0$ a constant}\\
  \sigma_s(t) &= \varphi(t+\gamma_0) \, , 
 \end{aligned}
\end{align}
valid even for the value $n=0$. 
The question arises which of the two effects, the instability induced by strain softening or the stabilizing effect of strain-rate sensitivity, 
wins the competition, and whether a given initial nonuniformity can lead 
to unstable modes that grow faster than the uniform shear \eqref{eq:uss}. 
This question has been considered for a power law model,
\begin{equation} \label{eq:constitution2}
 \sigma = \varphi(\gamma)\gamma_t^n = \gamma^{-m}\gamma_t^n, \quad m,n>0 \, ,
\end{equation}
in  \cite{HN77, molinari_analytical_1987, tzavaras_strain_1991}.
The associated system of partial differential equations becomes
\begin{equation} \label{eq:system}
 \begin{aligned}
   v_t & =\big(\gamma^{-m} v_x^n\big)_x \\
   \gamma_t&=v_x, 
 \end{aligned}
\end{equation}
where $n > 0$ and $0 < m < 1$. 
The stability of the uniform shearing solutions for this system, under velocity boundary conditions, is considered
in \cite{tzavaras_plastic_1986, tzavaras_strain_1991} via methods of nonlinear analysis, and in 
\cite{clifton_critical_1984, fressengeas_instability_1987, molinari_analytical_1987, tzavaras_nonlinear_1992} 
via linearized analysis techniques. 
In the region $n>m$, the uniform shear is both linearly and nonlinearly stable.
By contrast, linearized instability appears in the region $n < m$. Throughout the rest of this work, we are interested in the instability regime; hence, we focus on the parameter  range $n < m \le 1$, $n > 0$ small, and study
the behavior in the nonlinear regime.


\section{Focusing Self-similar Solutions} \label{sec:fss}
In this section, we study a family of focusing self-similar solutions for \eqref{eq:system} in the parameter regime $0 < n < m \le 1$.
Similar techniques were introduced in \cite{katsaounis_emergence_2014}, where the authors studied a thermally softening model,
and in the companion paper \cite{KLT_2016} providing a numerical construction valid in the special case $m=1$.

We begin by investigating the scale invariance properties of the system \eqref{eq:system}. For a given $(\gamma,v)$, we define $(\gamma_\rho$, $v_\rho)$ and the new independent variables $y$ and $s$ such that
\begin{equation} \label{eq:scale_inv}
\begin{aligned}
 &\gamma_\rho(t,x) = \rho^a \gamma(\rho^{-1}t, \rho^\lambda x), \quad
 v_\rho(t,x) = \rho^b v(\rho^{-1}t, \rho^\lambda x),\\
 &s=\rho^{-1}t, \quad
 y=\rho^ \lambda x.
\end{aligned}
\end{equation}
Due to the choice $\lambda>0$, this transformation makes the profile narrower and higher for $\rho$ large. This is of course in accordance with our goal, constructing solutions that localize the initial profile. A simple calculation shows that \eqref{eq:system} is invariant under \eqref{eq:scale_inv} if the exponents $a$ and $b$ are selected as 
\begin{equation} \label{eq:s_inv}
 a^{\lambda,m,n}= \frac{2-n}{1+m-n} + \frac{2}{1+m-n}\lambda, \quad b^{\lambda,m,n}= \frac{1-m}{1+m-n} + \frac{1-m+n}{1+m-n} \lambda.
\end{equation}

This motivates us to consider a family of self-similar solutions of the focusing type
\begin{equation} \label{eq:ss}
\begin{aligned}
 \bar{ \gamma }(t,x) &= (t+1)^a\bG((t+1)^ \lambda x),\\
 \bar{v}(t,x)&=  (t+1)^b\bV((t+1)^ \lambda x)
\end{aligned}
\end{equation}
with $\xi=(t+1)^ \lambda x$ is the similarity variable. Substitution into \eqref{eq:system} gives a system of ordinary differential equations
\begin{equation} \label{eq:sseqns}
 \begin{split}
 a^{\lambda,m,n}\bG + \lambda \xi \bG_\xi &= \bV_\xi,\\
 b^{\lambda,m,n}\bV + \lambda \xi \bV_\xi &= \big(\bG^{-m}\bV_\xi^n\big)_\xi.
 \end{split}
\end{equation}
where $a$, $b$ are given by \eqref{eq:s_inv}. 
We supplement the above equations with suitable initial conditions
\begin{align}
\label{ssdata}
 \bG(0) = \bG_0  > 0 \, , \quad \bV_{\xi} (0)  = \bU(0) = \bU_0  > 0 \, , 
 \end{align}
where $\bG_0$ and $\bU_0$ are positive parameters. As the problem is singular it is not a-priori clear how many conditions are needed;
the choice \eqref{ssdata} is justified by the analysis of the singularity at $\xi = 0$ presented below.
Given a  smooth solution of \eqref{eq:sseqns}, \eqref{ssdata} for parameters $\lambda$, $m$ and $n$, it will generate the profile  of a solution to \eqref{eq:system} that localizes 
at the focusing rate $\lambda$. Note that such a solution will be generated by an initial profile
$(\bG(x), \bV(x))$ at $t = 0$ and $\bG_0$, $\bU_0$ can be thought as measuring the size of the initial nonuniformity.

The system \eqref{eq:sseqns} is non-autonomous and singular at $\xi=0$. 
Existence of smooth solutions for such singular systems is not guaranteed by general theories and is effected via a case-by-case analysis. 
In the present case, it is possible to de-singularize \eqref{eq:sseqns}, turning it into an autonomous system of three differential equations. 
In a second step the problem is turned into the construction of a heteroclinic orbit for an equivalent system. 
The existence of the heteroclinic orbit is achieved in Section \ref{sec:proof} by employing geometric singular perturbation theory.

We reduce the problem on the right half plane $\xi \ge 0$.  Since the system is invariant under the transformation $\xi \rightarrow -\xi$, $\bG \rightarrow \bG$ and $\bV \rightarrow -\bV$, if we construct a smooth solution in the right-half-plane, then the even extension of $\bG$ and the odd extension of $\bV$ will give rise to a solution on the entire line. The conditions
\begin{equation} \label{eq:continuity}
 \frac{d}{d\xi}\bG(0)=\bV(0)=0
\end{equation}
are imposed to ensure $\bG$ and $\bV$ are smooth at $\xi=0$.

The system \eqref{eq:sseqns} has its own scale invariance. For a given $\big(\bG(\xi), \bV(\xi)\big)$, we define
\begin{equation}
 \bG_A(\xi) = A^\alpha \bG(A\xi), \quad \bV_A(\xi) = A^\beta \bG(A\xi).
\end{equation}
The exponents $\alpha$ and $\beta$ that make \eqref{eq:sseqns} invariant are
\begin{equation}
\alpha = \frac{-2}{1+m-n}, \quad \beta = - \frac{1-m+n}{1+m-n}. \label{eq:alpha_beta}
\end{equation}

Motivated by the previous observation, we introduce a change of variables
\begin{align} \label{eq:tvars}
 \begin{aligned}
 \bG(\xi) &= \xi^\alpha \tg(\log \xi), \quad \bV(\xi) = \xi^\beta \tv(\log \xi), \\
 \bU(\xi) &= \bV_\xi(\xi) = \xi^\alpha \tu(\log\xi), \quad \bS(\xi) =\bG^{-m}\bU^n= \xi^{-\alpha(m-n)}\ts(\log\xi),
 \end{aligned}
\end{align}
where $\eta = \log \xi$, $\eta \in (-\infty,+\infty)$ is the new independent variable.
Substitution into \eqref{eq:sseqns} gives the system for the residual variables $(\tg,\tv,\tu,\ts)$
\begin{equation}
 \begin{split}
 \frac{2-n}{1+m-n} \tg + \lambda \tg_\eta &= - \frac{1-m+n}{1+m-n}\tv + \tv_\eta,\\
 \frac{1-m}{1+m-n} \tv + \lambda \tv_\eta &= \frac{2(m-n)}{1+m-n}\ts + \ts_\eta,\\
 \ts &=\tg^{-m}(\beta \tv + \tv_\eta)^n. 
 \end{split}
\end{equation}
The third equation can be written as
$$\big(\ts \tg^m\big)^\frac{1}{n} = - \frac{1-m+n}{1+m-n}\tv + \tv_\eta\;(=\tu).$$
when $n>0$.
After rearrangement, we arrive at an autonomous system of three first-order equations
\begin{equation} \label{eq:auto0}
 \begin{split}
  \lambda \tg_\eta &= -\frac{2-n}{1+m-n}\tg + (\ts \tg^m)^{ \frac{1}{n} }, \\
  \tv_\eta &= \frac{1-m+n}{1+m-n}\tv + (\ts \tg^m)^{ \frac{1}{n} }, \\
  \ts_\eta &= - \frac{2(m-n)}{1+m-n}\ts + b^{ \lambda,m,n}\tv + \lambda(\ts \tg^m)^{ \frac{1}{n} }.
 \end{split}
\end{equation}


Further inspection shows that the variables cannot simultaneously equilibrate. For example, if $\tg \rightarrow \tg_\infty$ as $\eta \rightarrow \infty$, then $(\ts \tg^m)^{ \frac{1}{n} } \rightarrow \frac{2-n}{1+m-n}\tg_\infty$ and this makes $\tv$ diverge. 
This raises analytical difficulties addressed and we avoid them by introducing a second non-linear transformation so that the variables simultaneously equilibrate. 
One observation is that if $f \sim \xi^\rho$ as $\xi \rightarrow \infty$ (resp. as $\xi \rightarrow 0$), then $\partial_\eta(\log f) = \frac{\xi \partial_\xi f}{f} \rightarrow \rho$ 
 as $\xi \rightarrow \infty$ (resp. as $\xi \rightarrow 0$). 
Thus, if we  identify two quantities $f$ and $g$ that share the same asymptotic leading order as $\xi \rightarrow \infty$ (resp. as $\xi \rightarrow 0$), then $\partial_\eta(\log f) - \partial_\eta(\log g) = \partial_\eta(\log \frac{f}{g}) \rightarrow 0$,  in other words $\log \frac{f}{g}$ equilibrates as $\xi \rightarrow \infty$ (resp. as $\xi \rightarrow 0$). 
Using \eqref{eq:sseqns}, heuristic calculations can be carried out to find which three pairs of quantities are expected to share the same asymptotic leading order. 
This suggests to introduce the variables
\begin{equation}  
\label{eq:ratios}
 p= \frac{\xi^2\bG}{\bS}=\frac{ \tg}{ \ts}, \quad
 q= b^{\lambda,m,n}\frac{\xi\bV}{\bS}=b^{\lambda,m,n}\frac{ \tv}{\ts}, \quad
 r=\frac{\bU}{\bG} = \frac{\tu}{\tg} = \frac{ \big(\ts \tg^m\big)^\frac{1}{n} }{ \gamma }. 
\end{equation}
Substitution into \eqref{eq:auto0} leads to the equivalent system
{\small
\begin{align} 
 \dot{p} &=p\Big( ~~~~~~\frac{1}{ \lambda }\big(r - \frac{2-n}{1+m-n}\big) - \frac{1-m+n}{1+m-n} &+&1-q- \lambda p r\Big), \nonumber \\
 \dot{q} &=q\Big(                                                                          &+&1-q- \lambda p r\Big) + b^{\lambda,m,n}pr,\tag*{$(P)$\textsuperscript{$\lambda,m,n$}}\label{eq:pqrsystem}\\
 n\dot{r}&=r\Big( \frac{m-n}{ \lambda }\big(r - \frac{2-n}{1+m-n}\big) + \frac{1-m+n}{1+m-n} &-&1+q+ \lambda p r\Big), \quad   \nonumber
\end{align}
}
{where $\dot{(\cdot)} = \frac{d}{d{\eta}}(\cdot)$}.
The remainder of the work is organized as follows:
\begin{itemize}
\item[(a)] In section \ref{sec:dyn}, we show that the problem of existence of  a smooth profile $(\bG(\xi), \bV(\xi)$ satisfying  \eqref{eq:sseqns} and \eqref{ssdata}
can be reformulated to the construction of a suitable heteroclinic orbit for the system \eqref{eq:pqrsystem}.
\item[(b)] In section \ref{sec:proof}, we use geometric theory of singular perturbations (see appendix \ref{sec:geom}) to construct the heteroclinic orbit.
\end{itemize}

\begin{remark}
It is instructive to examine the relation between the uniform shearing solutions \eqref{intro-uss} (where for simlicity $\gamma_0 = 1$) 
and the focusing self-similar solutions \eqref{eq:ss}. If we select $\lambda = \tfrac{m-1}{2}$, then \eqref{eq:s_inv} implies $a = 1$,  $b = -\lambda$.
The function $\bV_s (\xi) = \xi$, $\bG_s  (\xi ) = 1$ solves \eqref{eq:sseqns} and the associated function emerging from \eqref{eq:ss},
$$
\bar \gamma (t,x) = (t+1) = \gamma_s (t) \, , \quad \bar v(t,x) = (t+1)^{-\lambda} \big ( (t + 1)^\lambda x \big ) = v_s (x)
$$
is precisely the uniform shearing solution. It corresponds however to a choice of parameter $\lambda < 0$ and it is not
of the focusing type.
\end{remark}

\section{Analysis of the Dynamical System} \label{sec:dyn}

In this section, we carry out the following steps: 
\begin{enumerate}
 \item[(i)] In section \ref{sec:equi}, we find the equilibria $M_i^{\lambda,m,n}$, $i=0,1,2,3$, of \eqref{eq:pqrsystem} and compute their local eigenstructure.
 \item[(ii)] In section \ref{sec:char}, we single out a heteroclinic orbit having the expected behavior as $\eta \rightarrow \pm \infty$. 
\end{enumerate}

\subsection{Equilibria and linear stability} \label{sec:equi}
The system \eqref{eq:pqrsystem} has four equilibrium points in the first octant of the phase space. The four equilibrium points are
\begin{align*}
 M_0^{ \lambda,m,n} 
     = \begin{pmatrix} 
        0 \\ 0 \\ a
       \end{pmatrix}, \quad
 M_1^{ \lambda,m,n} 
     = \begin{pmatrix} 
         0 \\  1  \\ c 
       \end{pmatrix}, \quad
 M_2^{ \lambda,m,n} 
     = \begin{pmatrix} 
        0 \\ 1 \\ 0
       \end{pmatrix}, \quad
 M_3^{ \lambda,m,n} 
     = \begin{pmatrix} 
        0 \\ 0 \\ 0
       \end{pmatrix},    
\end{align*}
where $a$ is the exponent in  \eqref{eq:s_inv} and 
\begin{equation}
\label{defc}
c =\frac{2-n}{1+m-n} - \frac{1-m+n}{(1+m-n)(m-n)} \lambda \, .
\end{equation}
Further, we define the constants:
\begin{align*}
 d &=\frac{1-m}{1+m-n} + \frac{2}{1+m-n} \lambda , \quad \quad
 e =\frac{1-m}{1+m-n} - \frac{2(m-n)}{1+m-n} \lambda,\\
 f &=\frac{1-m}{1+m-n} - \frac{1-m+n}{(1+m-n)(m-n)} \lambda, \quad \quad
 g =\frac{2-n}{1+m-n} + \frac{1-m+n}{1+m-n} \lambda,\\
 h &=\frac{2-n}{1+m-n} - \frac{2(m-n)}{1+m-n} \lambda,
\intertext{and}
A&=\left(\frac{m-n}{n}\right) \frac{a}{ \lambda },\quad 
 B=\left(\frac{m-n}{n}\right) \frac{c}{ \lambda },\quad 
 C=\left(\frac{m-n}{1-m+n}\right) B. 
\end{align*}
that are used to express the eigenvalues of the associated linearized problems.
Note that the constants $A$, $B$, and $C$ diverge as $n \rightarrow 0$. 

The analysis below applies to the case $n>0$, where we have eigenvalues $\mu_{i3}$, $i=0,1,2,3$ that are of $O(\frac{1}{n})$.
It is clear however that when $n=0$, the last equation of \eqref{eq:pqrsystem} becomes algebraic equation and the orbits are restricted on  the pieces of surface the equation specifies. There is no chance to escape the surface, i.e., the asymptotic structure around the equilibrium point is essentially of two dimensions. 

\eqref{fig:equilibria0} depicts the four equilibrium points in the first octant Arrows indicate the stable and unstable subspaces of each equilibrium point. When $\lambda > \frac{(2-n)(m-n)}{1-m+n}$, $c$ becomes negative and $M_1$ lies below the plane $r=0$. We will only be interested in the cases where $M_1$ lies above the plane $r=0$ and only in the region $r>0$. Thus, the case $c\le0$ is excluded from our study and accordingly $\lambda$ has the upper bound
\begin{equation} \label{eq:upperbdd}
 0< \lambda < \frac{(2-n)(m-n)}{1-m+n}.
\end{equation}
%
Next, when $ \lambda = \frac{2-n}{2(m-n)}$, $M_3$ is replaced by a line of equilibria which is the $p$-axis. Because it takes place on the plane $r=0$, this case is not treated separately. 

Now, we present the linear stability analysis of the equilibria. We denote the three eigenvalues of the linearization at each $M_i$ by $\mu_{ij}$ and the associated eigenvectors by $\vec{X}_{ij}$, $j=1,2,3$. The eigenvalues and eigenvectors as well as the equilibrium points are functions of $ \lambda$, $m$, and $n$. We omit this dependency for better readability but we will use superscripts when we need a clear distinction.
\begin{figure}
  \centering
  \psfrag{x0}{\scriptsize $M_0$}
  \psfrag{x1}{\scriptsize $M_1$}
  \psfrag{x2}{\scriptsize $M_2$}
  \psfrag{x3}{\scriptsize $M_3$}
  \psfrag{p}{\scriptsize $p$}
  \psfrag{q}{\scriptsize~~~$q$}
  \psfrag{r}{\scriptsize$r$}
  \psfrag{q*}{}
  \psfrag{r*1}{}
  \psfrag{r*2}{}
  \includegraphics[width=6cm]{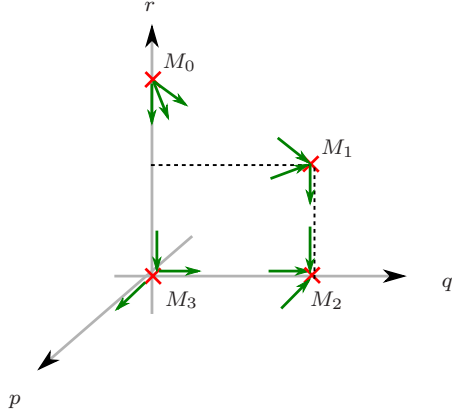} \label{fig:equilibria0}
  \caption{Equilibria of the $(p,q,r)$-system and the associated linearized vector fields for $\lambda$ satisfying \eqref{eq:upperbdd} and for $\mu_{31}>0$. }
\end{figure}

$\bullet$ $M_0$ \emph{is an unstable node}: all eigenvalues are real and positive; the first and the last eigenvectors lie on the $(q,r)$-plane, 
the first one is pointing right and down on the $(q,r)$-plane and the last one is pointing towards the origin on the $r$-axis. 
The second eigenvector is going off the $(q,r)$-plane. 
\begin{align*}
  \vec{X}_{01} &= \Bigg(0 , 1 , - \frac{\lambda}{m-n} \bigg(\frac{1}{1- A^{-1}} \bigg)\Bigg), \quad \mu_{01}=1,\\
  \vec{X}_{02} &= \Bigg( 1 , ab , - \frac{\lambda ad}{m-n} \bigg(\frac{1}{1- 2A^{-1}} \bigg)\Bigg), \quad \mu_{02}=2,\\
  \vec{X}_{03} &= (0 , 0 , 1), \quad \mu_{03}= A.
\end{align*}
We require $n$ to be sufficiently small so that $A^{-1}$ is small, $1-A^{-1}>0$ and $1-2A^{-1}>0$.

$\bullet$ $M_1$ \emph{is a saddle }: all eigenvalues are real, two of them are negative and one is positive. The unstable eigenspace,
$$  \vec{X}_{13} = (0, 0, 1), \quad \mu_{13}= B > 0 \, , $$
points towards the equilibrium $M_2$ and lies on the $(q,r)$-plane. 
There are two negative eigenvalues $\mu_{11}=-1$ and $\mu_{12}=- \frac{1-m+n}{m-n}$ and  the associated stable eigenspace  is two dimensional. 
Note that  $\mu_{12} = \mu_{11}$ when $m-n= \frac{1}{2}$. We specify the subspaces below:
\begin{enumerate}
 \item If $e=\frac{1-m}{1+m-n} - \frac{2(m-n)}{1+m-n}\lambda =0$, 
 \begin{align*} 
  \vec{X}_{11} &= \Bigg( 0, 1, - \frac{\lambda}{m-n} \bigg(\frac{1}{1+ B^{-1}} \bigg)\Bigg), \quad
  \vec{X}_{12} = 
\bigg(  1\;,\;0\;,\;  -\frac{\lambda}{m-n} \frac{\lambda c}{1+C^{-1}}\bigg).
  \end{align*}
  \item If $e\ne0$ and $m-n = \frac{1}{2}$, then $\mu_{11}=\mu_{12}=-1$ but its geometric multiplicity is $1$. It has the eigenvector $\vec{X}_{11}$, and the generalized eigenvector $\vec{X}_{12}'$ such that
  \begin{align*} 
  \vec{X}_{11} &= \Bigg( 0, 1, - \frac{\lambda}{m-n} \bigg(\frac{1}{1+ B^{-1}} \bigg)\Bigg), \quad 
  \vec{X}_{12}'= \Bigg(1 \;,\; -\lambda c - 2\lambda n \frac{e}{1+B^{-1}} \;,\; 0\Bigg).
  \end{align*}
  \item If $e\ne0$ and $m-n \ne \frac{1}{2}$,
  \begin{align*} 
  \vec{X}_{11} &= \Bigg( 0, 1, - \frac{\lambda}{m-n} \bigg(\frac{1}{1+ B^{-1}} \bigg)\Bigg), \\
  \vec{X}_{12} &= \Bigg(  \frac{-\frac{1-m+n}{m-n}+1}{ec}\;,\;1\;,\; -\frac{\lambda}{m-n}\frac{\big(-\frac{1-m+n}{m-n}+1\big)\lambda + e}{e(1+C^{-1})}\Bigg).
  \end{align*}
\end{enumerate}
The first eigenvector is in the $(q,r)$-plane and points up towards $M_1$ from the left. 
The second eigenvector or generalized eigenvector points towards $M_1$ coming from a direction off the $(q,r)$-plane.

$\bullet$ $M_2$ \emph{is a stable node}: all eigenvalues are real and negative. The eigenvectors are the coordinate basis vectors.
\begin{align*}
  \vec{X}_{21} &= (1,0,0), \quad \mu_{21}=- \frac{g}{ \lambda},\\
  \vec{X}_{22} &= (0,1,0), \quad \mu_{22}=-1,\\
  \vec{X}_{23} &= (0,0,1), \quad \mu_{23}=-B.
\end{align*}

$\bullet$ $M_3$ \emph{is at the origin and is a saddle}: all eigenvalues are real, the second eigenvalue is positive and the last eigenvalue is negative. 
The first eigenvalue changes sign at $\bar\lambda=\frac{2-n}{2(m-n)}$; it is negative if $\lambda<\frac{2-n}{2(m-n)}$ and positive if $\lambda>\frac{2-n}{2(m-n)}$. The eigenvectors are the coordinate basis vectors.
\begin{align*}
  \vec{X}_{31} &= (1,0,0), \quad \mu_{31}=- \frac{h}{ \lambda},\\
  \vec{X}_{32} &= (0,1,0), \quad \mu_{32}=1,\\
  \vec{X}_{33} &= (0,0,1), \quad \mu_{33}=-A.
\end{align*}

\subsection{Characterization of a suitable heteroclinic orbit} \label{sec:char}
Since there are four equilibria, the unstable or stable manifolds of each equilibrium point may conceivably intersect in various ways
and produce multiple heteroclinic connections. 
Our goal is to identify a  heteroclinic connection that provides a meaningful (from the perspective of mechanics) self-similar solution. 
Its characterization comes from analyzing the expected behavior as $\eta \rightarrow \pm\infty$. 

\subsubsection{Behavior at $+\infty$}
The profile of the targeted solution should correspond mechanical loading in the  shearing direction, and the resulting strain should be an
increasing function of time at any spatial point. For example, the strain of the uniform shearing solutions \eqref{eq:uss} grows linearly in time. 
We expect that the strain that is physically desirable for our solution should grow at a polynomial order.  Note, that if  $\gamma \sim t^\rho$, the quantity $\frac{t {\gamma}_t}{{\gamma}} \rightarrow \rho$ as $t \rightarrow \infty$. The quantity $r$ is
\begin{align*}
 r=\frac{\tu}{\tg} = \frac{\xi^{-\alpha} \bU(\xi)}{\xi^{-\alpha} \bG(\xi)} = \frac{\bU(\xi)}{\bG(\xi)}
 =\frac{t^{-(b+ \lambda)}\bar{v}_x}{t^{-a}\bar{\gamma}} = \frac{t \bar{\gamma}_t}{\bar{\gamma}}
\end{align*}
and thus we expect $r$ as $\eta \rightarrow \infty$ to tend to a (strictly) positive value. Among the equilibria $M_i$, $i=0,1,2,3$, $M_0$ is an unstable node, 
so we find $M_1$ as the only possibility that can provide the desired behavior. Thus, we select $M_1$ as the target of the desired heteroclinic as $\eta \to \infty$.

\subsubsection{Behavior at $-\infty$}\label{behminusinf}
The boundary conditions \eqref{eq:continuity} provide the desired behavior as $\eta \to -\infty$. 

\begin{proposition} \label{lem:alpha} Let $\big(\bG(\xi),\bV(\xi)\big)$ be a smooth solution of \eqref{eq:sseqns} 
with boundary conditions \eqref{eq:continuity}, and let $\big(p(\eta),q(\eta),r(\eta)\big)$ be the associated orbit obtained via the
transformations \eqref{eq:tvars} and \eqref{eq:ratios}. Then
\begin{align}
   e^{-2\eta}\left[\begin{pmatrix}
   p(\eta) \\ q(\eta) \\ r(\eta)
  \end{pmatrix}
  - M_0
  \right] \rightarrow \kappa\vec{X}_{02}, \quad \text{as $\eta \rightarrow -\infty$} \label{eq:estim}
\end{align} 
for the constant $\kappa = {\bar\Gamma(0)^{1+m-n}a^{-n}} >0$.
\end{proposition}
\begin{proof}
Let us compute the Taylor expansions of variables $p(\log\xi)$, $q(\log\xi)$ and $r(\log\xi)$ near $\xi=0$. 
The values of the variables and their derivatives evaluated at $\xi=0$ can be inferred by \eqref{eq:continuity} and by differentiating the system \eqref{eq:sseqns} repeatedly. 
A straightforward but cumbersome calculation yields
\begin{align*}
 \bar{V}(0)&=0,\quad
 \bar{V}_\xi(0) = \bar{U}(0) = \bar{ \Gamma }(0) a, \quad 
 \bar{ \Sigma }(0) = \frac{ \bar{U}^n }{ \bar{ \Gamma }^m }(0) = \bar{ \Gamma }(0)^{-(m-n)} a^n, \\
 \bar{\Gamma}_{\xi}(0) &= \bar{\Sigma}_{\xi}(0) = \bar{U}_\xi(0) = 0, \\
 \bar{\Gamma}_{\xi\xi}(0) &=-\big(\bG(0)^{2+m-n}a^{-n}\big) \frac{ \lambda ad }{m-n}  \Bigg(\frac{1}{ 1 -  2 A^{-1}}\Bigg)  \frac{1}{ \lambda }, \\
 \bar{U}_{\xi\xi}(0) &=-\big(\bG(0)^{2+m-n}a^{-n}\big) \frac{ \lambda ad}{m-n}  \Bigg(\frac{a+2 \lambda}{ 1 -  2  A^{-1}}\Bigg)  \frac{1}{ \lambda } ,
\end{align*}
and the leading orders of the Taylor expansion of the field variables near $\xi=0$ are 
\begin{align*}
 \bV(\xi) &= \bU(0)\xi + o(\xi^2),\\
 \bG(\xi) &= \bG(0) + \frac{\xi^2}{2} \bG_{\xi\xi}(0) + o(\xi^2), \\
 \bS(\xi) &= \bS(0) + \frac{\xi^2}{2} \bS_{\xi\xi}(0) + o(\xi^2), \\
 \bU(\xi) &= \bU(0) + \frac{\xi^2}{2} \bU_{\xi\xi}(0) + o(\xi^2).
\end{align*}

Using \eqref{eq:tvars} and \eqref{eq:alpha_beta}, we calculate the Taylor expansion of $p(\log\xi)$ near $\xi=0$, 
\begin{align*}
 {p}(\log\xi) &= 
  \frac{ \tilde \gamma}{ \tilde \sigma}  = 
   \frac{ \xi^{ \frac{2}{1 + m -n} } \, \bar\Gamma }{ \xi^{  - (m-n) \frac{2}{1 + m -n} } \, \bar\Sigma }   
  = \xi^2 \frac{ \bar\Gamma }{ \bar\Sigma }
  \\
 &= \xi^2 \frac{ \bar\Gamma(0) }{ \bar\Sigma(0) } + o(\xi^2) = \bG(0)^{1+m-n}a^{-n}\xi^2 + o(\xi^2). 
\end{align*}
{Also, the Taylor expansion of $q(\log\xi)$ near $\xi=0$ is given by}
\begin{align*}
 {q}(\log\xi) &= b\frac{ \tilde v}{ \tilde \sigma} 
 =      \frac{  b \,  \xi^{\frac{1-m+n}{1+m-n}} \bar\Gamma}{  \xi^{  - (m-n) \frac{2}{1 + m -n} } \, \bar\Sigma  }   
 = b\xi \frac{ \bar V }{ \bar\Sigma }
 \\
 &= b\bigg[\xi \frac{ \bar V(0) }{ \bar\Sigma(0) } + \xi^2  \bigg(\frac{\bar{ V }_\xi(0)}{\bar{ \Sigma }(0)} - \frac{ \bar V (0) \bar{ \Sigma }_\xi(0)}{\bar{ \Sigma^2 }(0)}\bigg) + o(\xi^2)\bigg]\\
 &=\big(\bG(0)^{1+m-n}a^{-n}\big)ab\xi^2 + o(\xi^2). 
\end{align*}
{Finally, $r(\log\xi)=a$ at $\xi=0$ and the Taylor expansion of $r(\log\xi)-a$ near \\$\xi=0$ yields, }
\begin{align*}
 {r}(\log\xi) -a &= \frac{ \tilde u}{ \tilde \gamma} - a
 =    \frac{ \xi^{ \frac{2}{1 + m -n} } \, \bar U }{ \xi^{ \frac{2}{1 + m -n} } \, \bar\Gamma } - a   
 = \frac{ \bar U }{ \bar\Gamma }(\xi) - \frac{ \bar{U}(0)}{ \bar{ \Gamma }(0)}\\
 &= \xi \frac{\bar{ U }(0)}{\bar{ \Gamma }(0)} \bigg(\frac{\bar{ U }_\xi(0)}{\bar{ U }(0)} - \frac{\bar{ \Gamma }_\xi(0)}{\bar{ \Gamma }(0)}\bigg) \\
 &+ \frac{1}{2}\xi^2\bigg[ \frac{\bar{U}_{\xi\xi}(0)}{ \bar{ \Gamma }(0)} - 2 \frac{ \bar{U}_\xi(0)\bar \Gamma_\xi(0)}{\bar \Gamma^2(0)} + \bar{U}(0) \bigg(- \frac{ \bar\Gamma_{\xi\xi}(0) }{\bar \Gamma^2(0)} + 2 \frac{\big(\bar \Gamma_\xi(0)\big)^2}{ \bar \Gamma^3(0) }\bigg) \bigg] + o(\xi^2)\\
 &=\big(\bG(0)^{1+m-n}a^{-n}\big)\frac{ -\lambda ad }{m-n}  \Bigg(\frac{1}{ 1 -  2  A^{-1}}\Bigg)   \xi^2 + o(\xi^2).
\end{align*}
From $\log\xi = \eta$ and the eigenvector of the unstable node $M_0$, we conclude that
\begin{align} \label{eq:asym_alpha}
   &e^{-2\eta}\left[\begin{pmatrix}
   p(\eta) \\ q(\eta) \\ r(\eta)
  \end{pmatrix}
  - M_0\right]
  \rightarrow {\bar\Gamma(0)^{1+m-n}a^{-n}}\vec{X}_{02}, \quad \text{as $\eta \rightarrow -\infty$.}
\end{align}
\end{proof}

\subsubsection{Selection of the targeted heteroclinic orbit}\label{sec:summarize} 
The asymptotic behavior as $\eta \rightarrow \pm\infty$ suggests to look for a heteroclinic orbit joining $M_0$ to $M_1$ that emanates in the direction of $\vec{X}_{02}$. 
Recall that $M_0$ is an unstable node and that $M_1$ has two dimensions of stable eigenspace. We conjecture from that there is a surface $G \subset W^u(M_0) \cap W^s(M_1)$, the intersection of the unstable manifold of $M_0^{ \lambda, m,n}$ and the stable manifold of the equilibrium $M_1^{ \lambda, m,n}$.

Assuming the surface, there is a one-parameter family of heteroclinic curves joining $M_0$ to $M_1$. 
In the neighborhood of $M_0$, because $\mu_{01}=1<2=\mu_{02}$, all the curves meet $M_0$ tangentially to $\vec{X}_{01}$ except one; 
this exceptional curve meet $M_0$ tangentially to the eigenvector$\vec{X}_{02}$. 
In other words, the asymptotic behavior at $\eta=\pm\infty$ we established in section \ref{behminusinf} characterizes to consider this exceptional curve.

In an autonomous system, if $\varphi^\star(\eta)$ is a heteroclinic orbit then so is $\varphi^\star(\eta+\eta_0)$ for any constant $\eta_0$. This implies that there is one-parameter family of heteroclinics that share the same orbit in phase space. 
\eqref{eq:asym_alpha} indicates that $\bG(0)$, the constant factor of the self-similar strain profile is responsible for fixing the shift $\eta_0$. Let us be precise on this procedure.
Any orbit near $M_0$ has the asymptotic expansion
\begin{equation}
 \varphi(\eta) - M_0 = \kappa_1e^\eta \vec{X}_{01} + \kappa_2e^{2\eta} \vec{X}_{02} + \text{higher-order terms} \label{eq:expansion}
\end{equation}
as $\eta \rightarrow -\infty$ for some constants $\kappa_1$ and $\kappa_2$. Let $\begin{pmatrix} P(\eta)\\Q(\eta)\\R(\eta) \end{pmatrix}$ be a trajectory
emanating in the direction of $\vec{X}_{02}$ and $\bar\kappa_2$ be the corresponding constant ($\bar\kappa_1=0$ for this trajectory). 
Any shifted trajectory 
$$
 \varphi^\star(\eta) = \begin{pmatrix} p(\eta)\\q(\eta)\\r(\eta) \end{pmatrix} = \begin{pmatrix} P(\eta+\eta_0)\\Q(\eta+\eta_0)\\R(\eta+\eta_0) \end{pmatrix}
$$ 
reparametrizes the same orbit and satisfies the asymptotic behavior
$$e^{-2\eta}\left[\begin{pmatrix}
   P(\eta+\eta_0)\\Q(\eta+\eta_0)\\R(\eta+\eta_0)
  \end{pmatrix}
  - M_0\right]
  \rightarrow {\bar\kappa_2e^{2\eta_0}}\vec{X}_{02}, \quad \text{as $\eta \rightarrow -\infty$.}
$$
We select the shift $\eta_0$ so as to satisfy \eqref{eq:estim}, that is
$$\eta_0 = \frac{1}{2}\log \frac{\bG(0)^{1+m-n}a^{-n}}{\bar\kappa_2}.$$

\begin{remark}
In the course of the selection, it is the parameters $m$ and $n$ of the material law,  the focusing rate $\lambda$, and the shift $\eta_0$ 
that fixes the heteroclinic orbit. From the perspective of the application, we may let $\big(\bG(0),\bU(0))$ the initial size of nonuniformities in the strain and that in the strain rate be the primary parameters instead of $\lambda$ and $\eta_0$. $\lambda$ and $\eta_0$ are then determined by
\begin{equation}
 \lambda = \frac{1+m-n}{2}\Big(\frac{\bU(0)}{\bG(0)} - \frac{2-n}{1+m-n}\Big), \quad \eta_0 = \frac{1}{2}\log \frac{\bG(0)^{1+m}\bU(0)^{-n}}{\bar\kappa_2}, \label{eq:lambda_eta}
\end{equation}
where we exploited $\frac{\bU(0)}{\bG(0)} = a$. \eqref{eq:upperbdd} pinpoints the range  
\begin{equation} \label{eq:upperbdd2}
 \frac{2-n}{1+m-n} \;<\; \frac{\bU(0)}{\bG(0)} \;<\; \frac{2-n}{1-m+n},
\end{equation}
where admissible self-similar solutions are attained. In summary, for each given $(m,n)$, we are looking for a two-parameter family 
of self-similar solutions. The two parameters, the sizes of the initial nonuniformities $\big(\bG(0),\bU(0))$, need to take values
in the range  \eqref{eq:upperbdd2}. 
\end{remark}

\subsection{Asymptotics of the profile}
In a similar fashion to the analysis providing the behavior at $\xi = 0$, we can pursue the asymptotic expansion of the heteroclinic
near $M_1$ on $G$.  Similarly to \eqref{eq:expansion}, we have
\begin{enumerate}
 \item[(a)] if $m-n = \frac{1}{2}$ and $\lambda \ne 1-m$
 \begin{equation}
 \varphi^\star(\eta) - M_1 = \kappa'_1e^{-\eta} \vec{X}_{11} + \kappa'_2\eta e^{- \eta} \vec{X}_{12}' + \text{higher-order terms} \label{eq:expansion2_special}
 \end{equation}
 \item[(b)] otherwise 
   \begin{equation}
 \varphi^\star(\eta) - M_1 = \kappa'_1e^{-\eta} \vec{X}_{11} + \kappa'_2e^{- \frac{1-m+n}{m-n}\eta} \vec{X}_{12} + \text{higher-order terms} \label{eq:expansion2}
\end{equation}
as $\eta \rightarrow \infty$.
\end{enumerate}

The following Proposition collects the calculations on the asymptotic behaviors.
\begin{proposition}
 Let $\big(\bG(\xi),\bV(\xi),\bS(\xi),\bU(\xi)\big)$ be a smooth self-similar profile satisfying \eqref{eq:continuity} and let $\varphi^\star(\eta)=\big(p(\eta),q(\eta),r(\eta)\big)$ be the associated variables defined by \eqref{eq:tvars} and \eqref{eq:ratios}. Suppose $\varphi^\star(\eta)$ is the heteroclinic orbit of \eqref{eq:pqrsystem} that connects $M_0$ to $M_1$. Then, 
 \begin{enumerate}
  \item[(i)] At $\xi = 0$, the self-similar profiles satisfy the boundary conditions
    \begin{equation*}
    \bar{V}(0) = \bar\Gamma_\xi(0) = \bar\Sigma_\xi(0) = \bar{U}_\xi(0)=0, \quad \text{$\bar{ \Gamma}(0), \bar{ U}(0)$ are given parameters.}
  \end{equation*}
  \item[(ii)] The asymptotic behavior as $\xi \rightarrow 0$ is given by
  \begin{equation} \label{eq:ss_asymp0}
  \begin{aligned}
    \bV(\xi) &= \bU(0)\xi + O(\xi^3), & \bG(\xi) &= \bG(0) + O(\xi^2), \\
    \bS(\xi) &= \frac{\bU(0)^n}{\bG(0)^m}+ O(\xi^2), & \bU(\xi) &= \bU(0) + O(\xi^2).
  \end{aligned}
  \end{equation}
  \item[(iii)] The asymptotic behavior as $\xi \rightarrow \infty$ is	
  \begin{enumerate}
   \item if $m-n = \frac{1}{2}$ and $\lambda \ne 1-m$
  \begin{equation} \label{eq:ss_asymp1_special}
  \begin{aligned}
    \bV(\xi) &= O((\log\xi)^{-\frac{1}{3}}), &    \bG(\xi) &= O(\xi^{- \frac{1}{m-n}}(\log\xi)^{\frac{2}{3}}),\\
    \bS(\xi) &= O(\xi(\log\xi)^{-\frac{1}{3}}), &   \bU(\xi) &= O(\xi^{- \frac{1}{m-n}}(\log\xi)^{\frac{2}{3}}),
  \end{aligned}
  \end{equation}
   \item and for all other cases 
  \begin{equation} \label{eq:ss_asymp1}
  \begin{aligned}
    \bV(\xi) &= O(1), &    \bG(\xi) &= O(\xi^{- \frac{1}{m-n}}),\\
    \bS(\xi) &= O(\xi), &   \bU(\xi) &= O(\xi^{- \frac{1}{m-n}}).
  \end{aligned}
  \end{equation}
  \end{enumerate}
 \end{enumerate}
 
\end{proposition}
\begin{proof}
  $(i)$ and $(ii)$ were verified in the proof of  Proposition \ref{lem:alpha}. To show $(iii)$, the asymptotic behavior of $\varphi^\star(\eta)$ in a neighborhood of $M_1$ is investigated. The asymptotic behavior of a nonlinear problem near the hyperbolic equilibrium point $M_1$ is determined by the associated linearized problem. 
As the orbit $\varphi^\star(\eta)$ lies in the stable manifold of $M_1$,  its asymptotic behavior  in the neighborhood of $M_1$ is expressed by \eqref{eq:expansion2_special} for the case $m-n= \frac{1}{2}$ and $\lambda \ne 1-m$ and by \eqref{eq:expansion2} for the rest of cases. We first focus on the latter.

For the orbit $\varphi^\star(\eta)$, one can further exclude the possibility that the coefficient $\kappa'_2 = 0$ in \eqref{eq:expansion2}.
Indeed,  the plane $p=0$ is invariant for the dynamical system \eqref{eq:pqrsystem}, and the corresponding orbits on the plane $p=0$ can be calculated. 
In this case  the equation for $q$ decouples and one can integrate for $q$ and $r$. 
In particular, the heteroclinic that connects $M_0$ to $M_1$ on the plane $p =0$  can be calculated and has  $\kappa'_2 = 0$ for its coefficient. 
Since the orbit $\varphi^\star(\eta)$ ventures out of the plane $p=0$, we conclude that for this one $\kappa'_2\ne0$. 

In addition, observe that $p \rightarrow 0$ and the $p$-component of $\vec{X}_{11}$ vanishes, which implies 
$$p = \kappa'_2e^{- \frac{1-m+n}{m-n}\eta} + \text{higher-order terms}, \quad q\rightarrow 1, \quad r\rightarrow c\quad  as ~~\eta \rightarrow \infty.$$
The reconstruction formulas for $\big(\tv,\tg,\ts,\tu\big)$ from $\big(p,q,r\big)$ are 
\begin{equation}
\begin{aligned}
\tv &= \frac{1}{b}\Big( p^{-(m-n)}q^{1+m-n}r^n \Big)^{ \frac{1}{1+m-n}},&
\tg &= \Big( pr^n \Big)^{ \frac{1}{1+m-n}},\\
\ts &= \Big( p^{-(m-n)}r^n \Big)^{ \frac{1}{1+m-n}},&
\tu &= \Big( pr^{1+m} \Big)^{ \frac{1}{1+m-n}}.
\end{aligned}\label{eq:recon}
\end{equation}
Hence, we conclude
\begin{equation*}
  \tv \sim e^{ \frac{1-m+n}{1+m-n}\eta}, \quad \tg \sim e^{-\frac{1-m+n}{(m-n)(1+m-n)}\eta}, \quad   \ts \sim e^{\frac{1-m+n}{1+m-n}\eta}, \quad \tu \sim e^{-\frac{1-m+n}{(m-n)(1+m-n)}\eta} 
\end{equation*}
as $\eta \rightarrow \infty$ and then setting  $\xi = e^\eta$ and \eqref{eq:tvars} provides $(iii)$-(b).

For the special case when $m-n= \frac{1}{2}$ and $\lambda \ne 1-m$, with same reasoning $\kappa_2'\ne0$ in \eqref{eq:expansion2_special} and the $p$-component of $\vec{X}_{11}$ vanishes but we have
$$p = \kappa'_2 \eta e^{-\eta} + \text{higher-order terms}, \quad q\rightarrow 1, \quad r\rightarrow c\quad  as ~~\eta \rightarrow \infty.$$
Straightforward calculations again from \eqref{eq:recon} gives $(iii)$-(a).
\end{proof}

%
%
%

\section{Existence of the heteroclinic orbit} \label{sec:proof}
In the preceding section, we identified the heteroclinic of \eqref{eq:pqrsystem} on the hypothesized surface $G = W^u(M_0) \cap W^s(M_1)$. We are now in a position to apply geometric singular perturbation theory to achieve the surface and the heteroclinic orbit. Our goal is to prove the following theorem.
\begin{theorem}\label{thm:1} Let $\Lambda$ be a domain of the tuple $(\lambda,m,n)\in\mathbb{R}^3$ defined by 
 \begin{align}
  0< m \le 1 \quad&\text{(strain softening with $m \le 1$)}, \label{eq:a1}\\
  n>0 \quad&\text{(rate sensitivity)}, \label{eq:a2}\\
  -m+n<0 \quad&\text{(unstable regime)}, \label{eq:a3}\\
  0< \lambda < \frac{(2-n)(m-n)}{1-m+n} \quad&\text{(strain must be increasing)}. \label{eq:a4}
\end{align}
 For each $(\lambda,m,0) \in \Lambda$, there is $n_0( \lambda,m)$, such that for $n \in (0, n_0)$, $(\lambda,m,n) \in \Lambda$ and the system \eqref{eq:pqrsystem} admits a heteroclinic orbit joining equilibrium $M_0^{\lambda,m,n}$ to equilibrium $M_1^{\lambda,m,n}$ with the following property
 \begin{equation} 
     e^{-2\eta}\left[\begin{pmatrix}
   p(\eta) \\ q(\eta) \\ r(\eta)
  \end{pmatrix}
  - M_0\right] \rightarrow \kappa\vec{X}_{02}, \quad \text{for some $\kappa>0$}. \label{eq:a5}
 \end{equation}
\end{theorem}
In the rest of this section, we divide the proof into several steps:
\begin{enumerate}
 \item[(i)] In preparation, we specify the two-dimensional {\it critical manifold} $G^0$. We inspect its normal hyperbolicity.
 \item[(ii)] Via geometric singular perturbation theory we continue this invariant surface to $n>0$ attaining the surface $G^n$. 
 \item[(iii)] We apply the Poincar\'e-Bendixson theorem to show the existence of the target heteroclinic orbit on $G^n$.
\end{enumerate}

\subsection{Critical manifold} \label{sec:crt_mfd}
The system in {\it fast scale} with the independent variable $\tilde{\eta} = \eta/n$ is
{\small
\begin{align} 
 p^\prime &=np\Big( ~~~~~~\frac{1}{ \lambda }\big(r - \frac{2-n}{1+m-n}\big) - \frac{1-m+n}{1+m-n} &+&1-q- \lambda p r\Big), \nonumber \\
 q^\prime &=nq\Big(                                                                          &+&1-q- \lambda p r\Big) + nb^{\lambda,m,n}pr, \tag*{($\tilde{P}$)\textsuperscript{$\lambda,m,n$}}\label{eq:pqr_fast} \\
 r^\prime&=r\Big( \frac{m-n}{ \lambda }\big(r - \frac{2-n}{1+m-n}\big) + \frac{1-m+n}{1+m-n} &-&1+q+ \lambda p r\Big)\nonumber \\
 &=:f^{\lambda,m,n}(p,q,r), \nonumber
\end{align}
}

\noindent where we denoted $(\cdot)^\prime = \frac{d}{d\tilde{\eta}}(\cdot)$. 
The right-hand side of the equation on $r$ is denoted by $f^{\lambda,m,n}(p,q,r)$. We specify the {\it critical manifold} $G^{\lambda,m,0}$ in the below that is a compact subset of $\{(p,q,r)\;|\; f^{\lambda,m,0}(p,q,r)=0\}$. The latter set consists of the equilibria of the system $\tilde{(P)}^{\lambda,m,0}$. 

In the region $r>0$, one solves the algebraic equation $f^{\lambda,m,0}(p,q,r)=0$,
\begin{equation}
r=h^{\lambda,m,0}(p,q) = \frac{\frac{m}{ \lambda} \frac{2}{1+m}-\frac{1-m}{1+m} + 1 - q}{\frac{m}{ \lambda} + \lambda p}, \label{eq:hn0}
\end{equation}
from which we notice that the contour lines are straight lines; after rearranging, 
\begin{equation}
 q + \lambda \underbar{r}p = \frac{2m}{1+m}-\frac{m}{ \lambda } \big( \underbar{r} - \frac{2}{1+m} \big), \quad \text{for  $h^{\lambda,m,0}(p,q)=\underbar{r}$.} \label{eq:level}
\end{equation}
In view of \eqref{eq:level}, the contour lines in the $pq$-plane sweep out the first quadrant from the origin. See \eqref{fig:contour}. More precisely, the contour line passes the origin when $\underbar{r}=a^{ \lambda,m,0}$ at the same time as its lift in the $pqr$-space passes the equilibrium $M_0^{ \lambda,m,0}$. As $\underbar{r}$ decreases, the contour line intersects the $p$ and $q$ axes and becomes steeper. When $\underbar{r}$ reaches $c^{ \lambda,m,0}$, the contour line passes $(0,1)$ at the same time as its lift passes $M_1^{ \lambda,m,0}$. $\underbar{r}$ then further decreases to $0$. 

Note that the inequality \eqref{eq:a4} implies $c^{ \lambda,m,0}>0$. We let $T$ be the triangle enclosed by the $p$-axis, $q$-axis and the one contour line of \eqref{eq:level} with $0<\underbar{r} < c^{\lambda,m,0}$. We choose $D \supset\supset T$ whose compact closure $\bar{D}$ is strictly away from $r=0$ plane. The critical manifold for each $\lambda$ and $m$ is defined by 
\begin{equation}
 G^{\lambda,m,0} = \{(p,q,r) \in \bar{D} \;|\; r=h^{\lambda,m,0}(p,q)\}.
\end{equation}

\begin{figure}[h]
  \centering
  \subfloat[phase space]{
    \psfrag{p}{$p$}
    \psfrag{q}{$q$}
    \psfrag{r}{$r$}
    \psfrag{M0}{\hskip -70pt$M_0$}
    \psfrag{M1}{\hskip -45pt$M_1$}
    \label{fig:graph}\includegraphics[width=5cm]{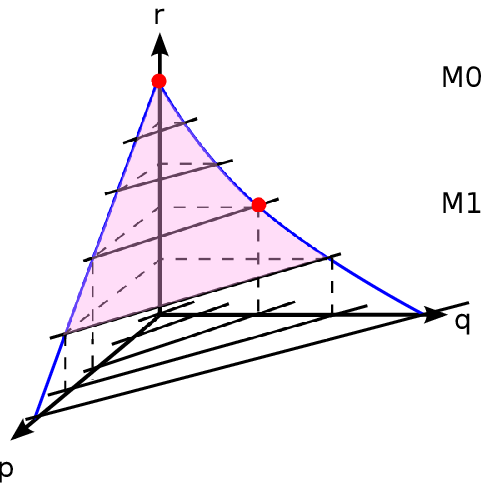}
  } \quad \qquad \qquad
  \subfloat[Contours on the $pq$-plane]{
    \psfrag{p}{$p$}
    \psfrag{q}{$q$}
    \psfrag{T}{\hskip 12pt$\mathbf{T}$}
    \psfrag{1}{}
    \psfrag{0}{\hskip -25pt$h^0(p,q)=a$}
    \psfrag{2}{\hskip -25pt$h^0(p,q)=c$}
    \psfrag{3}{\hskip -25pt$h^0(p,q)=\underbar{r}$}
    \psfrag{4}{\hskip -25pt$h^0(p,q)=0$}  
    \label{fig:contour}\includegraphics[width=5cm]{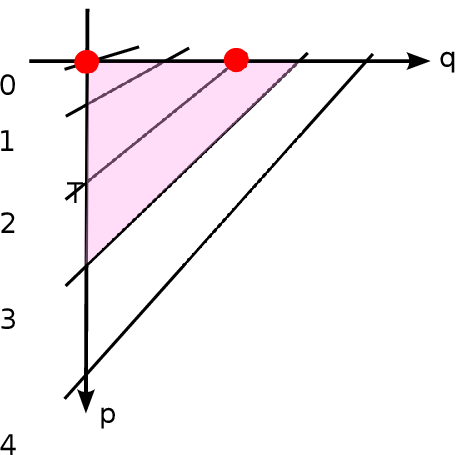}
  }
  \caption{Critical manifolds $G^{\lambda,m,0}=\big(p,q,h^{ \lambda,m,0}(p,q)\big)$, $0<m\le1$.}
  \label{fig:level}
\end{figure}

\begin{lemma} \label{lem:normal_hyper}
 $G^{\lambda,m,0}$ is a normally hyperbolic invariant manifold with respect to the system $(\tilde{P})^{ \lambda,m,0}$.
\end{lemma}
\begin{proof}
We linearize the system $(\tilde{P})^{ \lambda,m,0}$ around $G^{\lambda,m,0}$, and show that $0$ is the eigenvalue with a multiplicity of exactly $2$. Let the perturbations of $p$, $q$, and $r$ be $P$, $Q$, and $R$, respectively. After discarding  terms higher than the first order, we obtain
\begin{align*}
 \begin{pmatrix} {P}^\prime\\ {Q}^\prime \\ {R}^\prime \end{pmatrix} =
 \begin{pmatrix} 0 & 0& 0\\ 0 & 0 & 0\\ \lambda (h^{\lambda,m,0})^2 & h^{\lambda,m,0} & ( \frac{m}{ \lambda} + \lambda p )h^{\lambda,m,0} \end{pmatrix} \begin{pmatrix} {P}\\ {Q} \\ {R} \end{pmatrix}.
\end{align*}
The coefficient matrix has eigenvalues of $0$ and $( \frac{m}{ \lambda} + \lambda p )h^{\lambda,m,0}$. Since we take $h^{\lambda,m,0}$ away from zero and $p \ge 0$, the latter eigenvalue is strictly greater than zero. Thus, $0$ is an eigenvalue with multiplicity $2$. 
\end{proof}

\subsubsection{Flow on the critical manifold : the case $m=1$}

The marginal case $m=1$ provides closer detail. If one substitutes $m = 1$, $n = 0$, $h^{\lambda,1,0}(p,q)$ in the first two equations of \eqref{eq:pqrsystem}, the resulting system can be explicitly solved. The general solution on the graph is a family of parabolae $p=kq^2$ and $r=h^{\lambda,1,0}(p,q)$. This includes the two extremes $p=0$ and $q=0$, where $k$ takes $0$ and $\infty$ respectively. See \eqref{fig:hn0m1}. We focus on discussing two points: 1) In an effort to apprehend the flow of the rest of cases, we remark a few features for this marginal case, which in turn persist under the perturbation; and 2) we report features that do not persist too. These features do not play any role in our study, but this degenerate case is described here for clarity.

We address the first point. Emanating from $M_0^{ \lambda,1,0}$
in \eqref{fig:hn0m1_b} is a family of parabolae. Our interested direction $\vec{X}_{02}$ and the other $\vec{X}_{01}$ are indicated near $M_0^{ \lambda,1,0}$ by a dotted arrow. The family of parabolae is manifesting the fact that orbit curves meet $M_0^{ \lambda,1,0}$ tangentially to $\vec{X}_{01}$; one exception is the degenerate straight line that emanates in $\vec{X}_{02}$, which is depicted as the green one in \eqref{fig:hn0m1}, the target orbit. Another observation from the $pq$-plane is that the flow in the first quadrant far away from the origin is {\it inwards}. More precisely, as illustrated in \eqref{fig:hn0m1_b}, whenever $0<\underbar{r} < 1 = c^{\lambda,1,0}$ the flow on the contour line $\underbar{r} = h^{\lambda,1,0}$ is inwards. We make use of this observation in the proof of \eqref{sec:proof_proof}.

Now, we describe the second point. The crucial difference is that $M_1^{\lambda,1,0}$ is replaced by a line of equilibria $h^{\lambda,1,0}(p,q) = c^{\lambda,1,0}=1$, which is the red line in \eqref{fig:hn0m1}. As a result, each of the parabolae emanated from $M_0^{\lambda,1,0}$ lands at a point among these equilibria. $\vec{X}_{02}$ lies in the plane $q=0$ distinctively from all other cases and the target orbit in particular lands at the $q$-intercept of the line of equilibria. To compare this observation to the statement of \eqref{thm:1}, the target orbit does not connect $M_0^{ \lambda,1,0}$ to $M_1^{ \lambda,1,0}$ but to this $q$-intercept. This observation does not spoil our proof in \eqref{sec:proof_proof} because we assert the persistence of the critical manifold not the target orbit.
\begin{figure}[h]
  \centering
  \subfloat[phase space]{
    \psfrag{p}{$p$}
    \psfrag{q}{$q$}
    \psfrag{r}{$r$}
    \psfrag{M0}{\hskip -65pt$M_0$}
    \psfrag{M1}{\hskip -45pt$M_1$}
    \includegraphics[width=5cm]{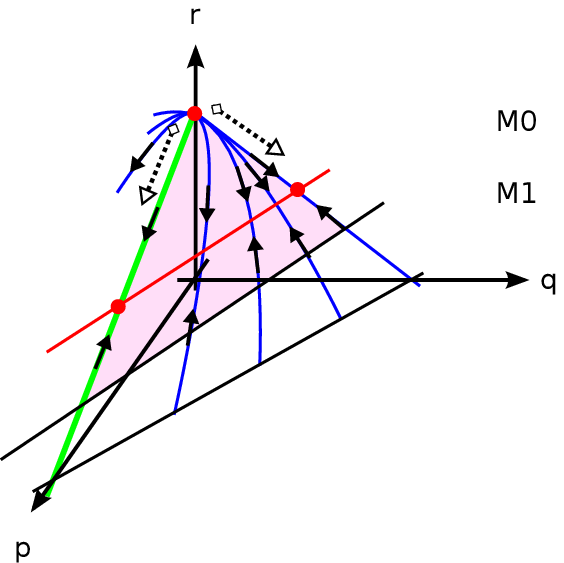}
  } 
  \subfloat[$pq$-plane]{
    \psfrag{p}{$p$}
    \psfrag{q}{$q$}
    \psfrag{T}{\hskip 12pt$\mathbf{T}$}
    \psfrag{2}{\hskip 0pt$h^0(p,q)=1$}
    \psfrag{3}{\hskip 0pt$h^0(p,q)=\underbar{r}$}
    \psfrag{X01}{$\vec{X}_{01}$}
    \psfrag{X02}{$\vec{X}_{02}$}
    \includegraphics[width=7.5cm]{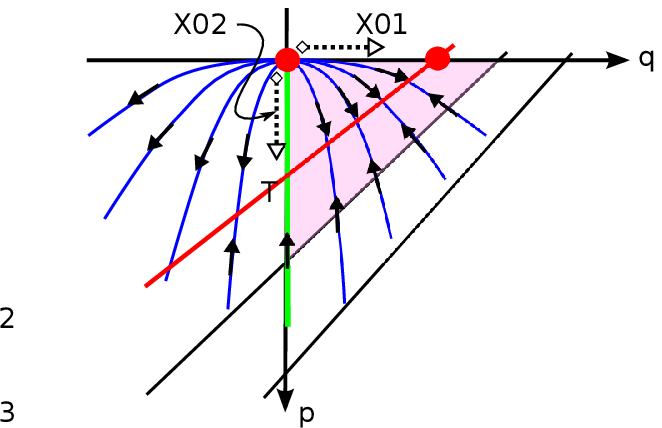} \label{fig:hn0m1_b}
  }
  \caption{Critical manifold $G^{\lambda,1,0}=\big(p,q,h^{ \lambda,1,0}(p,q)\big)$ when $m=1$ and the flow on the manifold.}\label{fig:hn0m1}
\end{figure}


\subsection{Proof of the theorem} \label{sec:proof_proof} \break

\smallskip
\noindent
\begin{proof}
\begin{step}
 Regularly perturbed reduced system.
\end{step}
By Lemma \ref{lem:normal_hyper} and Theorem \ref{thm:fenichel}, there exists $n_0$, such that for $n \in [0, n_0)$, locally invariant manifold $G^{\lambda,m,n}$ with respect to \eqref{eq:pqr_fast} exists. Moreover,   $G^{\lambda,m,n}$ is again given by the graph $(p,q,h^{\lambda,m,n}(p,q))$ on $\bar{D}$. The condition that $G^{\lambda,m,n}$ is disjoint from $r=0$ plane for all $n \in [0, n_0)$ must persist by making $n_0$ smaller if necessary. In addition, $n_0$ is chosen in the valid range of inequalities \eqref{eq:a3} and \eqref{eq:a4}.

After achieving $h^{\lambda,m,n}(p,q)$, substitution of the function in place of $r$ in system \eqref{eq:pqrsystem} leads to  the reduced systems that are parametrized by $\lambda$, $m$, and $n\in[0,n_0)$:
{\small
\begin{equation} \tag*{(${R}$)\textsuperscript{$\lambda,m,n$}} \label{eq:reduced}
\begin{split}
 \dot{p} &=p\Big(\frac{1}{ \lambda }\big(h^{\lambda,m,n}(p,q) - \frac{2-n}{1+m-n}\big) - \frac{1-m+n}{1+m-n} + 1-q- \lambda p h^{\lambda,m,n}(p,q)\Big),\\
 \dot{q} &=q\Big(                                                                          1-q- \lambda p h^{\lambda,m,n}(p,q)\Big) + b^{\lambda,m,n}ph^{\lambda,m,n}(p,q),
\end{split}
\end{equation}
}

\begin{step}
 $M_0^{\lambda,m,n}$ and $M_1^{\lambda,m,n}$ are still on the graph.
\end{step}
In fact, only $M_1^{ \lambda,1,n}$ needs to be checked because, other than that, the equilibrium points are hyperbolic. At $(p,q)=(0,1)$, from the system \eqref{eq:reduced}, we see $\dot{p} = \dot{q} = 0$. Now $\dot{r} = \frac{\partial h^{\lambda,1,n}}{\partial p} \dot{p} + \frac{\partial h^{\lambda,1,n}}{\partial q} \dot{q} = 0$ 
because the derivatives of $h^{\lambda,1,0}$ do not diverge and derivatives of $h^{\lambda,1,n}$ are close to them. This equilibrium point must be $M_1^{\lambda,1,n}$ since there is no other equilibrium point near $M_1^{\lambda,1,n}$. Similar reasoning in fact applies for the hyperbolic equilibrium points.


\medskip
\begin{step}
 $T$ is positively invariant under the flow \eqref{eq:reduced} if $n$ is sufficiently small.
\end{step}
First, we show the claim when $n=0$ and prove that it persists under the perturbation. Consider the system $(R)^{\lambda,m,0}$. On $p=0$, it is invariant; on $q=0$, the inward normal vector is $(0,1)$ and the inward flow $\dot{q} = b^{ \lambda,m,0}ph^{ \lambda,m,0} \ge 0$. On the hypotenuse contour line, if $\underbar{p}$ is the $p$-intercept and $\underbar{q}$ is the $q$-intercept, that is
$$ \underbar{q} = \frac{2m}{1+m}-\frac{m}{ \lambda } \big( \underbar{r} - \frac{2}{1+m} \big), \quad \underbar{p} = \frac{ \underbar{q} }{ \lambda \underbar{r} },$$
then $(-\underbar{q}, -\underbar{p})$ is an inward normal vector.
The inward normal component of the vector field on the line is then
\begin{align*} 
 (-\underbar{q}, &-\underbar{p}) \cdot ( \dot{p}, \dot{q} ) \\
 &=-\underbar{p}\underbar{q}(1-\underbar{q}) - p \frac{\underbar{q}}{m}\Big( \frac{2m}{1+m} - \frac{m}{ \lambda} \big( 1-\frac{2}{1+m} \big) - \underbar{q}\Big)\\
 &\ge -\underbar{p}\underbar{q}(1-\underbar{q}) \\
 &=: \delta >0.
\end{align*}
The inequality comes from $0<\underbar{r} < c^{ \lambda,m,0} \le 1$. $\delta$ is a fixed constant that is strictly positive, proving that the triangle $T$ is invariant.

\begin{figure}[h]
  \centering
  \subfloat[phase space]{
    \psfrag{p}{$p$}
    \psfrag{q}{$q$}
    \psfrag{r}{$r$}
    \psfrag{M0}{\hskip -65pt$M_0$}
    \psfrag{M1}{\hskip -30pt$M_1$}
    \includegraphics[width=5cm]{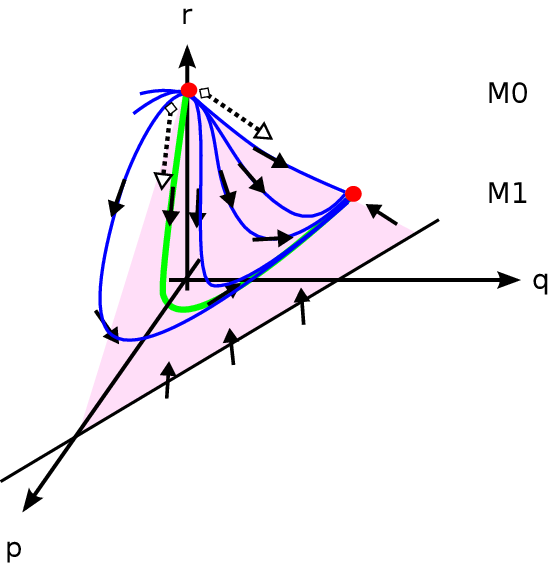}
    } \quad \quad \qquad
  \subfloat[$pq$-plane]{
    \psfrag{p}{$p$}
    \psfrag{q}{$q$}
    \psfrag{T}{\hskip 12pt$\mathbf{T}$}
    \psfrag{3}{\hskip -25pt$h^0(p,q)=\underbar{r}$}
    \psfrag{X01}{$\vec{X}_{01}$}
    \psfrag{X02}{$\vec{X}_{02}$}
    \includegraphics[width=5cm]{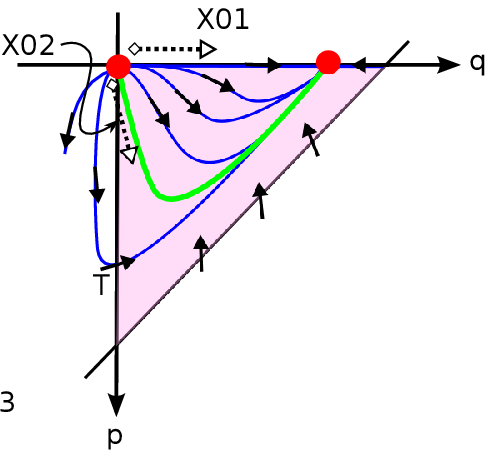}
  }
    \caption{Schematically drawn sketches of the perturbed invariant manifold and the flow. The triangle $T$ was determined by the contour line of the graph $\underbar{r}=h^0(p,q)$ but $T$ is positively invariant under the flow \eqref{eq:reduced} for all $n \in [0, n_0)$.}
\end{figure}


Now, we show that this positively invariant property persists under perturbation. We examine the same triangle $T$ but with  the system ${(R)}^{\lambda,m,n}$ with $n>0$. 
Again, sides of $p=0$ and $q=0$ are invariant or inward for the same reason. Now, the line of the hypotenuse of $T$ is no longer a contour line of $h^{ \lambda,m,n}(p,q)=\underbar{r}$, but $h^{ \lambda,m,n}(p,q)$ remains close to $\underbar{r}$, that is 
$$h^{ \lambda,m,n}(p,q) = \underbar{r} + ng_1(n,p,q), \quad \text{by Taylor theorem.}$$ 
{$g_1$ is uniformly bounded  in $n$, $p$, and $q$.} The inward normal component of the vector field on the line is computed as
\begin{align*} 
 (-\underbar{q}, &-\underbar{p}) \cdot ( \dot{p}, \dot{q} ) \\
 &= -\underbar{q}p\Big(\frac{1}{ \lambda }\big(h(p,q) - \frac{2}{1+m}\big) + \frac{2m}{1+m} -q- \lambda p h(p,q)\Big)-\underbar{p}q(1-q- \lambda p h(p,q)) \\&- \underbar{p}b^{\lambda,m,n}ph(p,q) \\
 &= -\underbar{q}p\Big(\frac{1}{ \lambda }\big(\underbar{r} - \frac{2}{1+m}\big) + \frac{2m}{1+m} -q- \lambda p \underbar{r}\Big)-\underbar{p}q(1-q- \lambda p \underbar{r}) - \underbar{p}b^{\lambda,m,0}p\underbar{r} \\ 
 &+n\Big(-\underbar{q}p\big( \frac{1}{ \lambda } g_1 - \lambda p g_1\big) - \underbar{p}q\big(- \lambda p g_1\big)\Big) - \underbar{p}\big( \frac{b^{\lambda,m,n}-b^{\lambda,m,0}}{n}p\underbar{r} + b^{\lambda,m,n} p g_1\big)\Big)\\
 &=-\underbar{p}\underbar{q}(1-\underbar{q}) - p \frac{\underbar{q}}{m}\Big( \frac{2m}{1+m} + \frac{m}{ \lambda} \big( \frac{2}{1+m}-1 \big) - \underbar{q}\Big)\\
 &+n\Big(-\underbar{q}p\big( \frac{1}{ \lambda } g_1 - \lambda p g_1\big) - \underbar{p}q\big(- \lambda p g_1\big)\Big) - \underbar{p}\big( \frac{b^{\lambda,m,n}-b^{\lambda,m,0}}{n}p\underbar{r} + b^{\lambda,m,n} p g_1\big)\Big)\\
 &\ge \delta + ng_2(n,p,q),
\end{align*}
where $g_2(n,p,q)$ is the expression in the parentheses of the last equality that is multiplied by $n$, which is also uniformly bounded in $n$, $p$, and $q$. We have used $ b^{\lambda,m,n}-b^{\lambda,m,0}=n\frac{(1-m) + 2 \lambda}{(1+m-n)(1+m)}$. 
Therefore, $n_0$ can be chosen, even smaller if necessary, so that the last expression becomes positive. This proves the claim. 
\medskip

Note that $\vec{X}_{02}$ is pointing inward of the triangle $T$ from $(0,0)$. Thus, the orbit emanating in $\vec{X}_{02}$ is continued to the interior of $T$ by the stable(unstable) manifold theorem. The $\omega$-limit set of this orbit cannot contain the limit cycle because when $n>0$, there is no equilibrium point inside of $T$ other than $(0,0)$ and $(0,1)$.  Recall that $(0,0)$ is the unstable node and $(0,1)$ generates the stable subspace. Thus, the Poincar\'e-Bendixson theory (for example in \cite{perko_differential_2001}) implies that the orbit converges to $(0,1)$.  The lifting of this orbit to the three dimensional phase space is the desired heteroclinic orbit. 
\end{proof}

\section{Two-parameter family of focusing solutions}
Theorem \ref{thm:1} determines the orbit curve, and the translation factor $\eta_0$ fixes the one heteroclinic orbit. As stated earlier, $\big(\bG(0),\bU(0)\big)$ determine $\eta_0$ and $\lambda$  by \eqref{eq:lambda_eta}. In summary, for each $\big(\bG(0),\bU(0)\big)$ such that
\begin{equation*}
 \frac{2-n}{1+m-n} \;<\; \frac{\bU(0)}{\bG(0)} \;<\; \frac{2-n}{1-m+n} \, ,
\end{equation*}
and for material parameters $m$ and $n$ this procedure gives rise to a solution of \eqref{eq:system}. By tracing back the nonlinear transformations \eqref{eq:ratios}, \eqref{eq:tvars}, and \eqref{eq:ss}
\begin{align}
 \begin{aligned}
 \gamma(x,t) &= (1 + t)^{ \frac{2-n}{1+m-n}+ \frac{ 2 }{1+m-n}\lambda } \;\bG\Big(x(1 + t)^{\lambda}\Big), \\
 v(x,t) &= (1 + t)^{ \frac{1-m}{1+m-n}+ \frac{ 1-m+n }{1+m-n}\lambda } \;\bV\Big(x(1 + t)^{\lambda}\Big), \\
 \sigma(x,t) &= (1 + t)^{ -\frac{2m-n}{1+m-n} - \frac{ 2(m-n) }{1+m-n}\lambda } \;\bS\Big(x(1 + t)^{\lambda}\Big), \\
 u(x,t) &=   v_x (x,t)  = (1 + t)^{\frac{1-m}{1+m-n}+ \frac{ 2 }{1+m-n}\lambda} \;\bU\Big(x(1 + t)^{\lambda}\Big), \\
 \end{aligned} \label{eq:sssol}
\end{align}
Note that the $\big(\bG,\bV,\bS,\bU\big)$ coincide with the initial nonuniformities of $\big(\gamma,v,\sigma,u\big)$ at $t=0$. 

We next describe the asymptotic behavior of the solutions. We omit the special case $m-n= \frac{1}{2}$ and $\lambda \ne 1-m$. We focus on the remaining cases;
in  the special case $m-n= \frac{1}{2}$ and $\lambda\ne1-m$, a  logarithmic correction will be required.

\medskip \noindent{\bf Initial nonuniformities $\big(\bG,\bV,\bS,\bU\big)$.} \medskip
We first look into the profiles $\big(\bG,\bV,\bS,\bU\big)$ of the initial nonuniformities. 
From \eqref{eq:ss_asymp0} and \eqref{eq:ss_asymp1} we infer that $\bG$ and $\bU$ peak at the origin and decay at the order $\xi^{ -\frac{1}{m-n}}$ as $\xi \rightarrow \infty$. 
They are thus bell-shaped even nonuniformities. On the other hand, $\bV$ is an odd function of $\xi$ connecting $-\bV_\infty$ to $\bV_\infty$ as $\xi$ runs from $-\infty$ to $\infty$, which describes the loading. $\bV_\infty= \lim_{\xi \rightarrow \infty} \bV(\xi)$ is a positive constant. We remark that $-\frac{1}{m-n} < -1$, and in the range of the parameters that we consider, $\bU=\bV_\xi$ is integrable. The stress $\bS=\bG^{-m}\bU^n$, an even function of $\xi$, has the local minimum $\bS(0)=\bG(0)^{-m}\bU(0)^n$ at $\xi=0$ and the asymptotic linear growth as $|\xi| \rightarrow \infty$.
\medskip

Below, we summarize the localizing behaviors of $\big(\gamma,v,\sigma,u\big)$. Due to the similarity structure $\xi=x(1+t)^\lambda$, the solution profiles shrink toward 
the origin as time proceeds. Second, 
due to the multiplier polynomials of $t$ in \eqref{eq:sssol}, the heights of the profiles increase (or decrease). The actual rate of growth or decay  at a fixed point $x\ne0$ is calculated by taking the shrinking effect into account, contrasting the behavior near the origin to the rest of the points.

\medskip {\bf$\bullet$ strain:}
We let the growth order $t^{ \frac{2-n}{1+m-n}}$ be critical. If $m=1$, it is of linear order and corresponds to that of the uniform shearing solution. If $0<m<1$, it is superlinear. We observe
\begin{align*}
 \gamma(0,t) &= (1+t)^{ \frac{2-n}{1+m-n}+ \frac{ 2 }{1+m-n}\lambda }\,\bG(0),\\
 \gamma(x,t) &\sim t^{ \frac{2-n}{1+m-n}- \frac{ 1-m+n }{(1+m-n)(m-n)}\lambda } |x|^{-\frac{1}{m-n}} \quad \text{as $t \rightarrow \infty$ for $x\ne 0$},
\end{align*}
and the tip of the strain $\gamma(0,t)$ grows supercritically and $\gamma(x,t)$, $x\ne 0$ grows subcritically as $t \rightarrow \infty$.

\medskip {\bf$\bullet$ strain rate:}
We observe
\begin{align*}
 u(0,t) &= (1+t)^{ \frac{1-m}{1+m-n}+ \frac{ 2 }{1+m-n}\lambda }\,\bU(0),\\
 u(x,t) &\sim t^{ \frac{1-m}{1+m-n}- \frac{ 1-m+n }{(1+m-n)(m-n)}\lambda } |x|^{-\frac{1}{m-n}} \quad \text{as $t \rightarrow \infty$ for $x\ne 0$},
\end{align*}
whose growth orders are by definition less by $1$ than those of the strain. The tip of the strain rate $u(0,t)$ certainly grows to $\infty$ as $t \rightarrow \infty$. Different from the strain, $u(x,t)$, $x\ne0$ does not necessarily grow as time proceeds. 

\medskip {\bf$\bullet$ velocity:}
The velocity $v(x,t)$ at a fixed time connects the $-v_\infty$ to $v_\infty$ as $x$ runs from $-\infty$ to $\infty$, where $v_\infty = \lim_{x \rightarrow \infty} v(x,t)$. This transition eventually becomes as drastic as the step function as $t \rightarrow \infty$. The limit value $v_\infty\sim t^b$.
\medskip

\medskip {\bf$\bullet$ stress:}
We observe
\begin{align*}
 \sigma(0,t) &= (1+t)^{ -\frac{2m-n}{1+m-n} - \frac{ 2(m-n) }{1+m-n}\lambda}\,\bG(0)^{-m}\bU(0)^n,\\
 \sigma(x,t) &\sim t^{ -\frac{2m-n}{1+m-n} + \frac{ 1-m+n }{1+m-n}\lambda } |x| \quad \text{as $t \rightarrow \infty$ for $x\ne 0$}.
\end{align*}
As the strain localizes near the origin, the stress at the origin collapses quickly. $\sigma(x,t)$, $x\ne0$ also decays to zero but at slower order; $-\frac{2m-n}{1+m-n} + \frac{ 1-m+n }{1+m-n}\lambda$ always is negative.

\appendix
\section{Geometric singular perturbation theory} \label{sec:geom}
In this section, we collected the part of geometric singular perturbation theory that we use. Our work  requires only  the basic technique;  readers are refered
to \cite{fenichel_persistence_1972,fenichel_geometric_1979,jones_geometric_1995,wiggins_normally_1994,KUEHN_2015} for further results and references.

The material below is taken from \cite{jones_geometric_1995,KUEHN_2015}. 
Let us consider the fast-slow system in the fast independent variable
\begin{equation} \label{eq:flow}
 \begin{split}
  x^\prime &=f(x,y,\epsilon),\\
  y^\prime &= \epsilon g(x,y,\epsilon)
 \end{split}
\end{equation}
and its critical case when $\epsilon=0$ 
\begin{equation} \label{eq:flow0}
 \begin{split}
  x^\prime &=f(x,y,0),\\
  y^\prime &= 0.
 \end{split}
\end{equation}
Here, $x\in \mathbb{R}^m$ is called a fast variable and $y\in \mathbb{R}^n$ a slow variable. We assume the vector field of \eqref{eq:flow} has a definition for $(x,y)\in U\subset\mathbb{R}^{m+n}$ an open set and for $\epsilon\in I$, the interval containing $0$. Let 
$$C_0:= \left\{ (x,y)\in U \;|\; f(x,y,0)=0\right\}.$$

To state a version of a theorem of geometric singular perturbation theory, we introduce two notions.
\begin{definition}[Normally hyperbolic invariant manifold in $C_0$ to \eqref{eq:flow0}] \label{def:nhim}
 A subset $S \subset C_0$ is called normally hyperbolic if the $m\times m$ matrix $(D_x f)(p,0)$ of first partial derivatives with respect to the fast variables $x$ has no eigenvalues with zero real part for all $p \in S$.
\end{definition}

\begin{definition}[Local invariance of \eqref{eq:flow}]
 Let $\phi_t(\cdot)$ denote the flow defined by the vector field of \eqref{eq:flow} and $M$ be a compact connected $C^\infty$-manifold with boundary embedded in $U$. $M$ is called a locally invariant manifold if for each $p\in M$, there exists a time interval $I_p=(t_1,t_2)$ such that $0\in I_p$ and $\phi_t(p) \in M$ for all $t\in I_p.$
\end{definition}
Next, we specify three hypotheses: 
\begin{enumerate}
 \item[(H1)] $f,g \in C^{\infty}(U \times I)$,
 \item[(H2)] The set $M_0 \subset C_0$ is a compact manifold, possibly with boundary, and is normally hyperbolic relative to \eqref{eq:flow0} in the sense of \eqref{def:nhim}.
 \item[(H3)] The set $M_0$ is given as the graph of the function $h^0(y) \in C^\infty(\bar{D})$ where $\bar{D} \subset \mathbb{R}^n$ is a compact simply connected domain with $C^\infty$ boundary.
\end{enumerate}
Below is a version of the theorem that is simpler in a sense that it involves a graph rather than a manifold. 
\begin{theorem}[Graph version, theorem 2 in \cite{jones_geometric_1995}] \label{thm:fenichel}
 Under the assumptions (H1),\\ (H2), and (H3), if $\epsilon>0$ is sufficiently small, there is a function $h^\epsilon(y)$, defined on $\bar{D}$, so that the graph $M_\epsilon = \{(x,y)~|~ x=h^\epsilon(y)\}$ is locally invariant to \eqref{eq:flow}. Moreover, $h^\epsilon$ is $C^r(\bar{D})$ for any $r<+\infty$, jointly in $y$ and $\epsilon$.
\end{theorem}

\section*{Acknowledgments}
This research was supported by King Abdullah University of Science and Technology (KAUST).

\bibliographystyle{siamplain}

\end{document}